\documentclass{article}
\usepackage{graphicx}
\usepackage[left=3cm, right=3cm]{geometry}
\usepackage[latin1]{inputenc}

\usepackage{color}
\usepackage{amsthm,amsmath}
\usepackage{amssymb}
\usepackage{amsfonts}
\usepackage{hyperref}
\hypersetup{colorlinks,allcolors=black}

\newtheorem{theorem}{Theorem}
\newtheorem{definition}{Definition}

\newtheorem{prop}{Proposition}
\newtheorem{lemma}{Lemma}
\newtheorem{example}{Example}

\author{Fatmanur Gursoy$^1$,  Elif Segah Oztas$^2$, Bahattin Yildiz$^3$ \\
	\small{$^1$Istanbul, Turkey}\\
	\small{$^2$Department of Mathematics, Karamanoglu Mehmetbey  University, Turkey}\\  
	\small {$^3$ Department of Math \& Stat, Northern Arizona University, Flagstaff, AZ, 86001, USA} \\
	\small{(fnurgursoy@hotmail.com, elifsegahoztas@gmail.com, bahattin.yildiz@nau.edu
		)}}

\begin{document}
\title{Reversible DNA codes over a family of non-chain rings $R_{k,s}$}% \footnote{A part of this study is presented in The International Conference on Coding theory and Cryptography (ICCC2015, Algeria).}}
\maketitle
\begin{abstract}
%In this paper we obtain reversible DNA codes by using skew polynomial rings  cyclic codes over $\mathbb{F}_{16}$,

In this paper, we solve the reversibility problem for DNA codes  over the non-chain ring $R_{k,s}=\mathbb{F}_{4^{2k}}[u_1,\ldots,u_{s}]/\langle u_1^2-u_1,\ldots, u_s^2-u_s\rangle$. We define an automorphism $\theta$ over $R_{k,s}$ which help us both find the idempotent decomposition of $R_{k,s}$ and solve the reversibility problem via skew cyclic codes. Moreover, we introduce a  generalized Gray map that preserves DNA reversibility.
\\
\textit{Keywords:} Skew Cyclic Codes, DNA Codes, Gray Map, Reversible Codes, DNA $k$-bases.\\
\textit{2010 Mathematics Subject Classification:} 94B15, 94B05,  92D20.
\end{abstract}
\section{Introduction}

The DNA consists of four bases (nucleotides), namely Adenine (A), Guanine (G), Cytosine (C) and Thymine (T). Those bases are lined up by the Watson-Crick complement (WCC) rule. WCC ensures that A pairs with T ($A^c=T$, $T^c=A$) and C pairs with G ($C^c=G$, $G^c=C$). 
Adleman, in \cite{adleman94}, used the structure of the DNA to solve a seven-node instance of the Hamiltonian Graph problem, an NP-complete problem similar to the travelling salesman problem. Since then, using DNA in other computational problems has become a focal point of research in many platforms. As further applications of this idea, for example in \cite{adleman99} and \cite{boneh}, the researchers used DNA to break the Data Encryption System. 

The DNA has also attracted the attention of researchers in working in Coding Theory, as it acts very much like an error-correcting code, when replicating. The work in DNA coding has been done using different approaches and tools-alphabets, such as DNA single basis (\cite{tehergf4, k1,tahersiapsim, siaptaherr}), double bases (\cite{agul,k2, ise, yildizsiap}), and multiple bases (\cite{k3,alger,  Gse, novel,kmer}) over various fields and recently, rings with complex structures. DNA codes have been studied in the context of  different constraints such as the Hamming distance constraint or the GC content constraint. In the works where multiple bases are matched with elements of an algebraic structure, the problem of reversibility arises, and this has been solved in the literature by using different methods.

The reversibility problem while working with algebraic structures that have more than four elements can be explained by the following example: Let $R_{1,1}$ be the ring $\mathbb{F}_{4^{2}}[u_1]/\langle u_1^2-u_1\rangle$ which has $16^2=4^4$ elements. To make a correspondence between the elements of $R_{1,1}$ and DNA multiple bases, we need to map each element of the ring to a DNA 4-bases. Let $(\alpha_1,\alpha_2,\alpha_3)$ be a codeword over $R_{1,1}$ corresponding to ATTCGAACTCCG (a $12$-string)
where $\alpha_1 \rightarrow $ATTC, $\alpha_2 \rightarrow $GAAC, $\alpha_3 \rightarrow  $TCCG and $\alpha_1 ,\alpha_2 ,\alpha_3 \in R_{1,1}.$  The reverse of $(\alpha_1,\alpha_2,\alpha_3)$ is $(\alpha_3,\alpha_2,\alpha_1)$, and $(\alpha_3,\alpha_2,\alpha_1)$
corresponds to TCCGGAACATTC. However, TCCGGAACATTC is not the reverse of ATTCGAACTCCG. Indeed, the reverse of ATTCGAACTCCG is GCCTCAAGCTTA. This deomnstrates that reversing the entries in the codeword does not lead directly to the reverse of the DNA sequence if we are working on an alphabet that has more than four elements. 

In this work, we focus on extending the problem introduced in \cite{dnahalka} and study  reversible DNA codes over $R_{k,s}=\mathbb{F}_{4^{2k}}[u_1,\ldots,u_{s}]/\langle u_1^2-u_1,\ldots, u_s^2-u_s\rangle$ where $u_iu_j=u_ju_i$.  To make a correspondence between the DNA sequences and the ring $R_{k,s}$, we map each ring element to a DNA $2^{s+1}k$-bases ($2^{s+1}k$-mer). We establish a generalized Gray map that preserves DNA reversibility. Moreover, we obtain good examples of reversible DNA codes from skew cyclic codes over $R_{k,s}$.

%The reason we consider the DNA k-bases (k-mers) is that components related to DNA, such as transcription factor, binding site, genes, proteins etc.,  are consist of k-bases. 
One motivation for considering a multi-base to element approach that we use in this work can be explained as follows. Recent studies have shown that components related to DNA, such as transcription factor, binding sites, genes, proteins etc.,  are composed of $k$-bases, i.e., they can be decomposed into $k$-base components that have meaningful interactions as opposed to single bases. Some of the recent research on the DNA has been focusing on finding the role of the specific $k$-bases. For example, in \cite{alper}, the DNA 8-bases (8-mers) are listed  from special areas including binding sites that have an important role in opening/closing the genes for producing proteins. In \cite{alper} (Table 6), the frequencies of 8-mers are listed and it has been demonstrated that the top area of the lists are important to identify binding sites where reversible-complement DNA 8-mers have been shown to appear intensively. Moreover, finding meaningful paths in DNA by using algebraic codes would potentially open up interesting research venues in understanding DNA and its related areas. 

The rest of the work is organized as follows. In Section 2, we give some of the preliminaries about the DNA and codes over a special family of rings. In Section 3, we describe a decomposition for the ring family $R_{k,s}$ to better understand its structure. In Section 4, we consider reversible DNA codes over $R_{k,s}$ together with some good examples of DNA codes obtained in this way. We finish with some concluding remarks and possible directions for future research. 

 \section{Preliminaries and definitions}
 In this section we  give  basic definition of skew cyclic codes and reversible codes and we  introduce skew cyclic codes over $R_{k,s}$. 
 Skew cyclic codes were first introduced by Boucher et al. in \cite{skew cyclic}. They  generalized  cyclic codes by using skew polynomial rings with an automorphism $\theta$ over the finite field with $q$ elements ($\mathbb{F}_q$). The definition of skew polynomial rings is given below.
 \begin{definition}\cite{Jacobson}
 	Let $R$ be a commutative ring with identity and $\theta$ be an automorphism over $R$. The set of polynomials $R[x;\theta]=\{a_0+a_1x+...+a_{n-1}x^{n-1}|a_i\in R, n\in \mathbb{N}\}$ is called the skew polynomial ring over $R$ where addition is the usual addition of polynomials and the multiplication is defined by $xa=\theta(a)x$ $(a\in R)$ and extended to polynomial multiplication naturally.
 \end{definition} 
 A skew cyclic code is defined to be a  linear code $C$ of length $n$ over $\mathbb{F}_q$ satisfying the property $\sigma(c)=(\theta(c_{n-1}),\theta(c_0),...,\theta(c_{n-2}))\in C$, for all $c=(c_0,c_1,...,c_{n-1})\in C$ \cite{skew cyclic}. In skew polynomial representation of  $C$, we can consider a codeword $c=(c_0,c_1,...,c_{n-1})$ of $C$ as a polynomial $c(x)=c_0+c_1x+\ldots+c_{n-1}x^{n-1}\in \mathbb{F}_q[x;\theta]$. In this case $\sigma(c)$
 corresponds to the polynomial $xc(x)$ in $\mathbb{F}_q[x;\theta]$. If the order of $\theta$, say $m$, divides $n$ then $C$ corresponds to a left ideal of the quotient ring $\mathbb{F}_q[x;\theta]/(x^n-1)$ (in \cite{skew cyclic}), else  $\mathbb{F}_q[x;\theta]/(x^n-1)$ is not a ring anymore but we can consider $C$ as a left $\mathbb{F}_q[x;\theta]-$submodule of $\mathbb{F}_q[x;\theta]/(x^n-1)$ (in \cite{I.siap}). In both cases $C$ is principally generated by a monic polynomial $g(x)$ which is a right divisor of  $x^n-1$ in $\mathbb{F}_q[x;\theta]$. One of the attractive  features of studying skew cyclic codes is that, factorization of $x^n-1$ is not unique in $\mathbb{F}_q[x;\theta]$, so we have much more generating polynomials in contrast with cyclic codes. %As a result of this feature,
 Thanks to their rich algebraic structure,  many new good codes with parameters exceeding the parameters of the previously best known linear codes have been obtained in both \cite{skew cyclic} and \cite{qc}.   There are many other studies done by considering skew polynomial rings over different rings instead of finite fields, such as skew cyclic codes over Galois rings (\cite{galois}),  finite chain rings (\cite{FCR}), non-chain rings such as   $\mathbb{F}_q+v\mathbb{F}_q$ (\cite{fq+vfq}) and $\mathbb{F}_q+u\mathbb{F}_q+v\mathbb{F}_q+uv\mathbb{F}_q$ (\cite{fquv}).

  In this paper, we study a family of rings 
  $R_{k,s}=\mathbb{F}_{4^{2k}}[u_1,\ldots,u_{s}]/\langle u_1^2-u_1,\ldots, u_s^2-u_s\rangle$, where $k,s\geqslant 1$ and $u_iu_j=u_ju_i$.  $R_{k,s}$ is a commutative  non-chain ring with characteristic $2$ and cardinality $4^{2k2^s}=4^{2^{s+1}k}$.

 We  use the following  automorphism to define skew cyclic codes over $R_{k,s}$. Let
 \begin{align*}
 \theta:&R_{k,s}\longrightarrow R_{k,s}\\
 &a\longrightarrow a^{4^k} \,\, \forall a\in \mathbb{F}_{4^{2k}},\\
 &u_i\longrightarrow u_i+1 \,\, \forall i\in\{1,\ldots,s\}.
 \end{align*}
Then $\theta$ is an automorphism over $R_{k,s}$  of order $2$. 
Naturally $R_{k,s}[x;\theta]$  is a skew polynomial ring and a skew cyclic code $C$ of length $n$ over $R_{k,s}$ corresponds to a left $R_{k,s}[x;\theta]$-submodule of $R_{k,s}[x;\theta]/(x^n-1)$. Whenever $g(x)$ is a right divisor of $x^n-1$ in $R_{k,s}$, the left $R_{k,s}[x;\theta]$-submodule  $C=(g(x))$ is a skew cyclic code over $R_{k,s}$. But not all skew cyclic codes over $R_{k,s}$ are principally generated.  In this study we focus on only principally generated ones. The factorization of $x^n-1$ is not unique in $R_{k,s}[x;\theta]$, since $\mathbb{F}_{4^{2k}}[x;\theta]$ is a subring of  $R_{k,s}[x;\theta]$ and the factorization of $x^n-1$ is not unique in  $\mathbb{F}_{4^{2k}}[x;\theta]$. So, there are many different right divisors of $x^n-1$. This algebraic property provides us an advantage for finding  right divisors   with special properties (palindromic and $\theta-$palindromic) in Section \ref{sec4}.
 \begin{definition}
 	Let $C$ be a code of length $n$ over $\mathbb{F}_q$. If $c^r=(c_{n-1},c_{n-2},...,c_1,c_0)\in C$ for all $c=(c_0,c_1,...,c_{n-1})\in C$, then $C$ is called a reversible code.
 \end{definition}
%\section{Application of Chinese Remainder Theorem to $R_{k_s}$}

\section{Decomposition of $R_{k,s}$}
 In this section, we introduce a method to find idempotent decomposition of the ring $R_{k,s}$. By using the decomposition of $R_{k,s}$ we define a Gray map which preserves DNA reversibility.
 
  Our aim is to establish a correspondence between the elements of $R_{k,s}$ and DNA $2^{s+1}k$-bases in such a way that the reversibility problem is solved by using skew cyclic codes over $R_{k,s}$. First of all we need to express the elements of $R_{k,s}$ in a unique way.
  % we will define a Gary map from $R_{k,s}$ to $\mathbb{F}_{4^{2k}}^{2^s}$. To do this we need to express the elements of $R_{k,s}$ in a unique way.
  Let us  arrange the elements of the form $u_1^{r_1}u_2^{r_2}\ldots u_s^{r_s}$ in an order. Let $A$ be the set of such elements, i.e. $A=\{u_1^{r_1}u_2^{r_2}\ldots u_s^{r_s}|r_i\in\{0,1\}, 1\leq i\leq s\}.$  The cardinality of $A$ is $2^s.$ Define a map
% Clearly; the following map is a one-to-one and onto map.
 \begin{align*}
 \mu: A&\longrightarrow \mathbb{Z}^s\\
 u_1^{r_1}u_2^{r_2}\ldots u_s^{r_s}&\longrightarrow  (r_1,r_2,\ldots,r_s).
 \end{align*}
 There are several ways of ordering  $ s $-tuples $(r_1,r_2,\ldots,r_s)$. We will use the lexicographic order.

 \begin{definition}[\cite{lex}]
 	Let $\alpha=(\alpha_1,\alpha_2,\ldots,\alpha_s)$ and $\beta=(\beta_1,\beta_2,\ldots,\beta_s)$ be in
 	$\mathbb{Z}_{\geqslant 0}^s$. If the leftmost nonzero entry of the vector difference $\alpha-\beta\in \mathbb{Z}^s$ is positive then $\alpha$ is greater than $\beta$ with respect to the lexicographic order, and denote it  $\alpha>_{lex}\beta$. We  write $u_1^{\alpha_1}u_2^{\alpha_2}\ldots u_s^{\alpha_s}>_{lex}u_1^{\beta_1}u_2^{\beta_2}\ldots u_s^{\beta_s}$ if $\alpha>_{lex}\beta$.
 \end{definition}
 For example; let us put $u_1,u_2\in R_{1,3}$ in order. Since  $\mu(u_1)=(1,0,0)$,   $\mu(u_2)=(0,1,0)$ and $(1,0,0)-(0,1,0)=(1,-1,0)$ then $u_1>_{lex}u_2$.

 %Before defining a Gray map we need to arrange the elements of $R_{k,s}.$
 %from $R_{k,s}$ to $\mathbb{F}_{4^{2k}}^{2}$

 Here, we rename the elements of $R_{k,s}$ of the form $u_1^{r_1}u_2^{r_2}\ldots u_s^{r_s}$ where $r_i\in{0,1}$ and $1\leq i\leq s$, as follows:

% Let $A=\{u_1^{r_1}u_2^{r_2}\ldots u_s^{r_s}|r_i\in\{0,1\}, 1\leq i\leq s\}.$ The cardinality of $A$ is $2^s.$
\begin{enumerate}
	\item First consider the following specific subsets of A;
	$$A_j=\{u_1^{r_1}u_2^{r_2}\ldots u_s^{r_s}|\text{exactly }j \text{ of $r_i$'s equal $1$ the others $0$}\},$$
	where $0\leq j \leq s.$ So, $A_0=\{1\}, \, A_1=\{u_1,u_2,\ldots,u_s\}, A_2=\{u_1u_2,u_2u_3\ldots,u_{s-1}u_s\}, A_s=\{u_1u_2\ldots u_s\},$ etc.
	\item Arrange the elements of $A_j$ with respect to the lexicographic order, and let $B_{j}$ be the  array consisting the ordered elements of $A_j$. For example $B_0=[1],$ $B_1=[u_1,\ldots, u_s],$ etc.
	\item Then %combine
	concatenate these arrays  as $B=[B_0,B_1,\ldots,B_s]$.
	\item Rename the $i$th element of the array $B$ as $U_i.$ For instance, $U_0=1,\ U_1=u_1,\  \ldots, \ U_{2^s-1}=u_1u_2\ldots u_s$.	
\end{enumerate}

Define the ordered $s$-tuples of $R_{k,s}$ as 
\begin{equation}
T=[\mu(U_0),\mu(U_1),\ldots,\mu(U_{2^s-1})]=[T_0,T_1,\ldots,T_{2^s-1}].
\end{equation}
%So, $T$
Note that; $\mu(U_i)+\mu(U_{2^s-1-i})=T_i+T_{2^s-1-i}=(1,1,\ldots,1)$ ,  for $0\leq i\leq2^s-1. $

\begin{example}\label{ex1}
	Let us consider the ring $R_{1,3}=\mathbb{F}_{16}[u_1,u_2,u_3]/\langle u_1^2-u_1,u_2^2-u_2, u_3^2-u_3\rangle.$
	Then $A=\{1,u_1,u_2,u_3,u_2u_3,u_1u_3,u_1u_2,u_1u_2u_3\}$ and
	\begin{align*}
	B_0=[1], \ \ \ B_1=[u_1,u_2,u_3], \\
	B_2=[u_1u_2,u_1u_3,u_2u_3], \ \ \ B_3=[u_1u_2u_3].
	\end{align*}
	So, $B=[1,u_1,u_2,u_3,u_1u_2,u_1u_3,u_2u_3,u_1u_2u_3]=[U_0,U_1,\ldots, U_7]$
	and the ordered $3-tuples$ of $R_{1,3} $ is; 
	$T=[T_0,\ldots,T_7]=[(0,0,0),(1,0,0),(0,1,0),(0,0,1),(1,1,0),(1,0,1),(0,1,1),(1,1,1)].$
	Any element $\alpha\in R_{1,3}$ can be written uniquely of the form $\alpha=a_0U_0+a_1U_1+\ldots a_7U_7$ where $\alpha_i\in \mathbb{F}_{16}$. Moreover, it can be easily seen that $T_i+T_{7-i}=(1,1,1)$, for example;
	\begin{align*}
	T_0+T_7=\mu(U_0)+\mu(U_7)=\mu(1)+\mu(u_1u_2u_3)=(0,0,0)+(1,1,1)=(1,1,1),\\
	T_3+T_4=\mu(U_3)+\mu(U_4)=\mu(u_3)+\mu(u_1u_2)=(0,0,1)+(1,1,0)=(1,1,1).
	\end{align*}
\end{example}

The following theorem is a result of the Chinese Remainder Theorem; 

\begin{theorem}\label{CRT}
	Let $I=\{I_i\}$ be a set of $2^s$ elements of $R_{k,s}$, where $0\leq i \leq 2^s-1$. If
	% the following three conditions are satisfied,
	\begin{itemize}
		\item[(i)] $I_i^2=I_i,$
		\item[(ii)] $I_iI_j=0$ for $i\neq j$ and $0\leq i,j \leq 2^s-1$,
		\item[(iii)] $\sum_{i=0}^{2^s-1} I_i=1,$
	\end{itemize}
	then $R_{k,s}=I_0\mathbb{F}_{4^{2k}}\oplus I_1\mathbb{F}_{4^{2k}}\oplus\ldots \oplus I_{2^s-1}\mathbb{F}_{4^{2k}}.$
\end{theorem}
	
Now we  construct an idempotent set $I$ satisfying these three conditions above. We associate some elements of $R_{k,s}$ with the s-tuple $T_i$'s as follows.
Let $T_i=(r_1,r_2,\ldots,r_s)$. Then define $ I_i=\theta^{r_1}(u_1)\theta^{r_2}(u_2)\ldots\theta^{r_s}(u_s)$. For example; since $T_0=(0,0,0,\ldots,0)$, $T_1=(1,0,0,\ldots,0)$ and $T_2=(0,1,0,\ldots,0)$ we have $I_0=u_1u_2\ldots u_s$, $I_1=(u_1+1)u_2\ldots u_s$ and $I_2=u_1(u_2+1)u_3\ldots u_s$.

\begin{theorem}\label{decom}
	Let $I=\{I_i=\theta^{r_1}(u_1)\theta^{r_2}(u_2)\ldots\theta^{r_s}(u_s)|\, T_i=(r_1,r_2,\ldots,r_s), 0\leq i \leq 2^s-1\}$. Then $R_{k,s}=I_0\mathbb{F}_{4^{2k}}\oplus I_1\mathbb{F}_{4^{2k}}\oplus\ldots \oplus I_{2^s-1}\mathbb{F}_{4^{2k}}.$
\end{theorem}
\begin{proof}
	\textbf{(i)} $I_i^2=I_i.$\\	
	Each $I_i$ has the form $(u_{j_1}+1)\ldots(u_{j_d}+1)u_{j_d+1}\ldots u_{j_s}$, where $\{j_1,j_2,\ldots, j_s\}=\{1,2,\ldots, s\}$ and $1\leq d \leq s $.  Since $(u_{j_e}+1)^2=(u_{j_e}+1)$ and $u_{j_f}^2=u_{j_f}$, then $I_i^2=I_i$.
		
		\textbf{(ii)} $I_iI_j=0$ for $i\neq j.$\\
		Let
		\begin{align*}
		I_i=\theta^{r_1}(u_1)\theta^{r_2}(u_2)\ldots\theta^{r_s}(u_s)\rightarrow T_i=(r_1,r_2,\ldots,r_s),\\
		I_j=\theta^{m_1}(u_1)\theta^{m_2}(u_2)\ldots\theta^{m_s}(u_s)\rightarrow T_j=(m_1,m_2,\ldots,m_s).
		\end{align*}
		Since $i\neq j$, then $T_i$ and $T_j$ are distinct so, at least one coordinate $r_e\neq m_e$, where $1\leq e \leq s$. Without loss of generality, we may assume that  $r_e=0$ and $m_e=1.$ Then $\theta^{r_e}(u_e)=u_e$ and $\theta^{m_e}(u_e)=u_e+1$. Since $u_e(u_e+1)=0$, we have $I_iI_j=0$.	
			
	\textbf{(iii)}	$\sum_{i=0}^{2^s-1} I_i=1.$
	
	The proof is by induction.
	For $s=0$, $I_0=1$.
	Suppose the induction hypothesis $\sum_{i=0}^{2^j-1} I_i=1$ is true for all $0<j<l$. Let us look at the case $s=l$.
	Exactly half of the elements of the set $I=\{I_0,I_1,\ldots, I_{2^l-1}\}$ have $u_l$ as a factor and the others have $(u_l+1)$ as another factor.

	If we take the subset of $I$ which has idempotents with the factor $u_l$, and delete $u_l$ from each element then we have the idempotent set of $R_{k,l-1}$, say $\{V_0,V_1,\ldots,V_{2^{l-1}-1}\}$ thus by induction we have 	$\sum_{i=0}^{2^{l-1}-1}V_i=1$.
	
	If we take the subset of $I$ which has idempotents with the factor $(u_l+1)$, and delete $(u_l+1)$ from each element then we have the idempotent set of $R_{k,l-1}$, again.
	
	So,
	\begin{align*}
	\sum_{i=0}^{2^l-1} I_i&=u_l\sum_{i=0}^{2^{l-1}-1}V_i +(u_l+1)\sum_{i=0}^{2^{l-1}-1}V_i\\
	&= u_l+(u_l+1)=1.
	\end{align*}	
\end{proof}
\begin{example}
	Let us determine the idempotent set $I$ for the ring $R_{1,3}$, given in Example \ref{ex1}.
	Since $T=[T_0,\ldots,T_7]=[(0,0,0),(1,0,0),(0,1,0),(0,0,1),(1,1,0),(1,0,1),(0,1,1),(1,1,1)]$, we have
	 \begin{align*}
	&I_0=u_1u_2u_3,\\
	&I_1=(u_1+1)u_2u_3,\\
	&I_2=u_1(u_2+1)u_3,\\
	&I_3=u_1u_2(u_3+1),\\
	&I_4=(u_1+1)(u_2+1)u_3,\\
	&I_5=(u_1+1)u_2(u_3+1),\\
	&I_6=u_1(u_2+1)(u_3+1),\\
	&I_7=(u_1+1)(u_2+1)(u_3+1).				
	\end{align*}	
	
\end{example}
By using the decomposition of $R_{k,s}$ in Theorem \ref{decom}, each element $\alpha\in R_{k,s}$ can uniquely be expressed as $\alpha=\alpha_0I_0+\alpha_1I_1+\ldots+ \alpha_{2^s-1}I_ {2^s-1},$   where $\alpha_0,\alpha_1,\ldots, \alpha_{2^s-1}\in \mathbb{F}_{4^{2k}}.$  By using this expression, we define the following Gray map:
 \begin{align*}\label{Gray}
 \varphi: R_{k,s}&\longrightarrow \mathbb{F}_{4^{2k}}^{2^s}\\
 \alpha&\longrightarrow (\alpha_0,\alpha_1,\ldots, \alpha_{2^s-1}).
 \end{align*}

 This Gray map is a one-to-one and onto map and can naturally be extended to $n$-tuples coordinatewise. The following propositon allows us to determine the Gray image of an element of the form  $\alpha=b_0U_0+b_1U_1+\ldots+b_{2^s-1}U_{2^s-1}$. The dot product  ``$\cdot$" below denotes the componentwise multiplication of two vectors, for example, $(0,1,0)\cdot (1,0,0)=(0,0,0).$
 \begin{prop}
 	Let  $\alpha=b_0U_0+b_1U_1+\ldots+b_{2^s-1}U_{2^s-1} \, \in R_{k,s}$ where $b_i\in \mathbb{F}_{4^{2k}}$ for $0\leq i \leq 2^s-1$. Then $\varphi(\alpha)=(\alpha_0,\alpha_1,\ldots,\alpha_{2^s-1})$ and
 	$$\alpha_i=\sum_j b_j \text{ where } 0\leq j\leq 2^{s}-1 \text{ and }j \text{ satisfies }T_i\cdot T_j=(0,0,\ldots,0).$$
 \end{prop}
 \begin{proof}
 	Let $\varphi(\alpha)=(\alpha_0,\alpha_1,\ldots,\alpha_{2^s-1})$. Then $\alpha=\alpha_0I_0+\alpha_1I_1+\ldots+ \alpha_{2^s-1}I_{2^s-1}$. Multiplying $\alpha$ by $I_i$ we have
 	\begin{align*}
 	I_i\alpha=\alpha_iI_i=I_i(b_0U_0+b_1U_1+\ldots+b_{2^s-1}U_{2^s-1})\\
 	\rightarrow \alpha_iI_i=b_0U_0I_i+b_1U_1I_i+\ldots+b_{2^s-1}U_{2^s-1}I_i
 	\end{align*}
 	since $I_iI_j=0$ for $i\neq j$.
 	Now let us determine $U_jI_i$. Let the ordered s-tuple of $I_i$ be $T_i=(r_1,r_2,\ldots,r_s)$ and the ordered s-tuple of $U_j$ be $T_j=(e_1,e_2,\ldots,e_s)$.
 	If $r_k=e_k=1$ for some $k\in\{1,2,\ldots,s\}$, then  $I_i$ has $(u_k+1)$ as a factor and $U_j$ has $u_k$ then $U_jI_i=0$. This means $b_j$ does not appear in the result of $I_i\alpha$. Otherwise, if either $r_k=e_k=0$ or $r_k\neq e_k$ for all $k\in\{1,2,\ldots, s\}$ then $U_jI_i\neq0$ so $b_j$ appears in the result of $I_i\alpha$.
 Whenever $r_k=e_k=0$ or $r_k\neq e_k$ for all $k\in\{1,2,\ldots, s\}$, $T_iT_j=(0,0,\ldots,0)$. Hence we can conclude that $\alpha_i=\sum b_j \text{ where } 0\leq j\leq 2^{s}-1 \text{ and }j \text{ satisfies }T_i\cdot T_j=(0,0,\ldots,0).$ 	
 \end{proof}
 % a map $\phi$ from  $R_{k,s}^n$ to $\mathbb{F}_{4^{2k}}^{2^sn}$ as follows, $\phi(c_0, c_1, \ldots, c_{n-1})=$

% In \cite{Gse}, Oztas and Siap gave a correspondence between the  elements of the field $\mathbb{F}_{4^{2k}}$ and DNA $2k-$bases. In order to solve the reversibility problem they mapped each $\beta$   and $\beta^{4^k}\in \mathbb{F}_{4^{2k}}$ to the DNA  $2k-$bases that are reverses of each other. They gave an algorithm to do this mapping.  Let us call this map with $\tau$. $\tau$ can naturally be extended to a map $\tau_2$ from $\mathbb{F}_{4^{2k}}^{2^s}$ to DNA $2^{s+1}k$- bases as follows;  $\tau_2(\alpha_0,\alpha_1,\ldots, \alpha_{2^s-1})=(\tau(\alpha_0),\tau(\alpha_1),\ldots,\tau(\alpha_{2^s-1}))$, where $\alpha_i\in \mathbb{F}_4^{2k}$, $0\leq i\leq 2^s-1$.
\section{Reversible DNA codes over $R_{k,s}$}\label{sec4}
In this section we obtain reversible DNA codes from skew cyclic codes over $R_{k,s}$ by making use of the Gray map defined in the previous section.

 In \cite{Gse}, Oztas and Siap gave a correspondence between the  elements of the field $\mathbb{F}_{4^{2k}}$ and DNA $2k-$bases. In order to solve the reversibility problem they mapped each $\beta$   and $\beta^{4^k}\in \mathbb{F}_{4^{2k}}$ to the DNA  $2k-$bases that are reverses of each other. They gave an algorithm to do this mapping.  Let us call this map with $\tau$.
 $\tau$ can naturally be extended to a map $\tau_2$ from $\mathbb{F}_{4^{2k}}^{2^s}$ to DNA $2^{s+1}k$- bases as follows;  $\tau_2(\alpha_0,\alpha_1,\ldots, \alpha_{2^s-1})=(\tau(\alpha_0),\tau(\alpha_1),\ldots,\tau(\alpha_{2^s-1}))$, where $\alpha_i\in \mathbb{F}_{4^{2k}}$, $0\leq i\leq 2^s-1$.

 In our study $\beta^{4^k}=\theta(\beta)$ for $\beta\in \mathbb{F}_{4^{2k}}$, thus $\tau(\beta)$ and $\tau(\theta(\beta))$ are reverses of each other. Briefly, we  say that $\beta$ and  $\theta(\beta)$ are DNA-reverses of each other.% and denote by $\beta^r=\theta(\beta).$

 	Let $\alpha\in R_{k,s}$ and  $\varphi(\alpha)=(\alpha_0,\alpha_1,\ldots, \alpha_{2^s-1})$. The corresponding DNA $2^{s+1}k$-bases of $\alpha$ is $\tau_2(\alpha_0,\alpha_1,\ldots, \alpha_{2^s-1})=(\tau(\alpha_0),\tau(\alpha_1),\ldots,\tau(\alpha_{2^s-1}))$ and its reverse is $$(\tau(\theta(\alpha_{2^s-1})),\tau(\theta(\alpha_{2^s-2})),\ldots, \tau(\theta(\alpha_{0}))).$$
In short, we say $(\alpha_0,\alpha_1,\ldots, \alpha_{2^s-1})$ and $(\theta(\alpha_{2^s-1}),\theta(\alpha_{2^s-2}),\ldots,\theta(\alpha_{0}))$ are DNA reverses of each other.

% To make the connection between the elements of $R_{k,s}$ and DNA $2^{s+1}k$- bases we  define the following map; $\Phi=\tau_2\circ\varphi$   ??biraz daha yaz

 \begin{lemma} $\theta(I_j)=I_{2^s-1-j}$ for all $j\in\{0,1,\ldots,2^s-1\} .$
 	\end{lemma}
 	\begin{proof}
 		Let $T_j=(r_1,r_2,\ldots,r_s)$. Since $T_j+T_{2^s-1-j}=(1,1,\ldots,1)$, then $T_{2^s-1-j}=(1-r_1,1-r_2,\ldots,1-r_s).$
 		Since $r_i$'s are either $0$ or $1$ and the order of $\theta$ is $2$, then $\theta^{r_i}=\theta^{-r_i}$. Thus,
 		$$I_j=\theta^{r_1}(u_1)\theta^{r_2}(u_2)\ldots\theta^{r_s}(u_s)\Rightarrow \theta(I_j)=\theta^{1-r_1}(u_1)\theta^{1-r_2}(u_2)\ldots\theta^{1-r_s}(u_s)=I_{2^s-1-j}.$$ 		
 	\end{proof}

Following theorem leads us to solve the reversibility problem directly by skew cyclic codes and it shows that  the Gray map  $\varphi$ preserves the DNA reversibility.

\begin{theorem}
	Let $\alpha\in R_{k,s} $. Then the DNA reverse of $\varphi(\alpha)$ is $\varphi(\theta(\alpha))$.
\end{theorem}	

\begin{proof}
	Let $\alpha=\alpha_0I_0+\alpha_1I_1+\ldots+ \alpha_{2^s-1}I_ {2^s-1}$. Then $\varphi(\alpha)=(\alpha_0,\alpha_1,\ldots, \alpha_{2^s-1})$, and
	 \begin{align*}
	 \theta(\alpha)&=\theta(\alpha_0)\theta(I_0)+\theta(\alpha_1)\theta(I_1)+\ldots+ \theta(\alpha_{2^s-1})\theta(I_ {2^s-1})\\
	&=\theta(\alpha_0)I_{2^s-1}+\theta(\alpha_1)I_{2^s-2}+\ldots+ \theta(\alpha_{2^s-1})I_ 0\\
	&=\theta(\alpha_{2^s-1})I_ 0+\theta(\alpha_{2^s-2})I_1+\ldots+\theta(\alpha_0)I_{2^s-1}.
	 \end{align*}

Thus $\varphi(\theta(\alpha))=(\theta(\alpha_{2^s-1}),\theta(\alpha_{2^s-2}),\ldots,\theta(\alpha_0) )$ which is exactly the DNA reverse of $\varphi(\alpha)$.\end{proof}

\begin{example}
	In \cite{ise}, Table-1 gives a correspondence between the elements of $\mathbb{F}_{16}$ and DNA 2-bases. It has similar properties with the correspondence algorithm given in \cite{Gse}. Table-1 maps each element of $\mathbb{F}_{16}$ and  its fourth power to the DNA 2-bases that are reverses of each other.
	
	 Consider the ring $R_{1,3}$, given in Example \ref{ex1}. Let $\beta$ be a primitive element of $\mathbb{F}_{16}$ and $\alpha=\beta^2I_0+\beta I_1+\beta^5 I_2+ \beta^{3}I_3+I_4+0I_5+\beta^7I_6+I_7 \in R_{1,3}$. Then, $\varphi(\alpha)=(\beta^2,\beta,\beta^5, \beta^3,1,0,\beta^7,1).$ Using Table-1 given in \cite{ise}, the corresponding DNA $16$-bases is 	 	
	 	\begin{align*}
	 	\tau_2(\beta^2,\beta,\beta^5, \beta^3,1,0,\beta^7,1)&=(\tau(\beta^2),\tau(\beta),\tau(\beta^5),\tau(\beta^3),\tau(1),\tau(0), \tau(\beta^7),\tau(1))\\
	 	&=(GC,AT,CC,AG,TT,AA,GT,TT).
	 	\end{align*}	 	
	 	Also,	 	
	 	\begin{align*}
	 	\theta(\alpha)&=I_0+\beta^{13} I_1+0 I_2+ I_3+\beta^{12}I_4+\beta^5I_5+\beta^4I_6+\beta^8I_7 \\
	 	&\Rightarrow \varphi(\theta(\alpha))=(1,\beta^{13},0,1,\beta^{12},\beta^5,\beta^4,\beta^8)\\
	 	&\Rightarrow \tau_2(1,\beta^{13},0,1,\beta^{12},\beta^5,\beta^4,\beta^8)
	 	=(TT,TG,AA,TT,GA,CC,TA,CG).
	 	\end{align*}	 	
	 	Thus, $\varphi(\alpha)$ and $\varphi(\theta(\alpha))$ are DNA reverses of each other.
\end{example}
The Gray map $\varphi$ can be extended to $n-$coordinates. For $c=(c_0,\ldots, c_{n-2}, c_{n-1})\in R_{k,s}^n$ we have $\varphi(c)=(\varphi(c_0),\ldots, \varphi(c_{n-2}), \varphi(c_{n-1}))$. In this case the DNA reverse of $\varphi(c)$ will be $\varphi(\theta(c)^r)=(\varphi(\theta(c_{n-1})),\varphi(\theta(c_{n-2})),\ldots\varphi(\theta(c_0)))$ where $\theta(c)=(\theta(c_0),\ldots, \theta(c_{n-2}), \theta(c_{n-1}))$.  
\begin{definition}
Let $C\subseteq  R_{k,s}^n$. If the DNA reverse of  $\varphi(c)$ exists in $\varphi(C)$, for all $c\in C,$ then $C$ or equivalently
$\varphi(C)$ is called a reversible DNA code.
\end{definition}
\begin{definition}[\cite{alger}]
	Let $f(x)=a_0+a_1x+\ldots +a_tx^t$ be a polynomial of degree $t$ over  $ R_{k,s}$.  $f(x)$ is said to be a palindromic polynomial if $a_i=a_{t-i}$  for all $i\in  \{0,1,\ldots,t \}$.
	And $f(x)$ is said to be a $\theta$-palindromic polynomial if $a_i=\theta(a_{t-i})$  for all $i\in \{0,1,\ldots,t\}$.
\end{definition}

Let $C$ be a skew cyclic code of length $n$ over $\mathbb{F}_q$ with respect to an automorphism $\theta'$. If the order of $\theta'$ and $n$ are relatively prime then $C$ is a cyclic code over $\mathbb{F}_q$ \cite{I.siap}. Similarly, any skew cyclic code of odd length over $R_{k,s}$ with respect to  $\theta$ is a cyclic code, since the order of $\theta$ is $2$. For this reason we restrict the length $n$ to even numbers only.
\begin{theorem}\label{teo1}
	Let $C=( g(x))$ be a skew cyclic code of length $n$  over $R_{k,s}$ where  $g(x)$ is a right divisor of $x^n-1$ in $R_{k,s}[x;\theta]$ and $deg(g(x))$ is odd. If  $g(x)$ is a $\theta$-palindromic polynomial then $\varphi(C)$ is a reversible DNA code.
\end{theorem}
	\begin{proof}
		Let $g(x)$ be a $\theta$-palindromic polynomial and $t=n-deg(g(x))$. Recall that $\varphi(c)$ and $\varphi(\theta(c)^r)$  are DNA reverses of each other.
		Note that, we do not distinguish between the vector representation and the polynomial representation of a codeword in $R_{k,s}^n$. For each $c\in C$, $c(x)=\sum_{i=0} ^{t-1}\beta_i x^i g(x)$ for some $\beta_i\in R_{k,s}$, since $C=(g(x))$. Then the DNA reverse of  $\varphi(c)$ is   $\varphi(c')$ where $c'(x)=\sum_{i=0} ^{t-1} \theta(\beta_i) x^{t-1-i}g(x)$. 
		
		 Since $ c'(x)=\sum_i \theta(\beta_i) x^{t-1-i}g(x)\in C$, then $\varphi(c')\in \varphi(C)$. Hence, $\varphi(C)$ is a  reversible DNA code.
	\end{proof}
\begin{theorem}\label{teo2}
	Let $C=( g(x))$ be a skew cyclic code of length $n$  over $R_{k,s}$ where  $g(x)$ is a right divisor of $x^n-1$ in $R_{k,s}[x;\theta]$ and $deg(g(x))$ is even. If  $g(x)$ is a palindromic polynomial then $\varphi(C)$ is a reversible DNA code.
\end{theorem}	
\begin{proof}
	Let $g(x)$ be a palindromic polynomial and $t=n-deg(g(x))$. For each $c\in C$, $c(x)=\sum_{i=0} ^{t-1}\beta_i x^i g(x)$ for some $\beta_i\in R_{k,s}$, since $C=(g(x))$. Then the DNA reverse of  $\varphi(c)$ is   $\varphi(c')$ where $c'(x)=\sum_{i=0} ^{t-1} \theta(\beta_i) x^{t-1-i}g(x)$. 
	
	Since $ c'(x)=\sum_i \theta(\beta_i) x^{t-1-i}g(x)\in C$, then $\varphi(c')\in \varphi(C)$. Hence $\varphi(C)$ is a  reversible DNA code.
\end{proof}
In Remark 1 of \cite{dnahalka}, an illustration of proofs has been given for the case $k=1$ and $s=2$, i.e. for the ring $R_{1,2}=\mathbb{F}_{16}+u_1\mathbb{F}_{16}+u_2\mathbb{F}_{16}+u_1u_2\mathbb{F}_{16}$. This simple method works for all rings $R_{k,s}$.
\begin{example}
	Let us consider the skew polynomial ring $R_{1,3}[x;\theta]$. Let $\beta$ be a primitive element of	$\mathbb{F}_{4^2}$, then
\begin{align*}
x^6-1=h(x)g(x)=&[1+(\beta(u_2+u_3)+\beta^{7})x+(\beta(u_2+u_3)+\beta^{7})x^2+x^3]\\
&[1+(\beta(u_2+u_3)+\beta^{7})x+(\beta^{4}(u_2+u_3)+\beta^{13})x^2+x^3] \,\, \in R_{1,3}[x;\theta].
\end{align*}
Since $g(x)=1+(\beta(u_2+u_3)+\beta^{7})x+(\beta^{4}(u_2+u_3)+\beta^{13})x^2+x^3$  is a $\theta$-palindromic polynomial with odd degree then $C=( g(x))$ is a reversible DNA-code over $R_{1,3}$.
	\end{example}
	\begin{example}
		Let us consider the skew polynomial ring $R_{1,3}[x;\theta]$.	Let $\beta$ be a primitive element of	$\mathbb{F}_{4^2}$, then
		\begin{align*}
		x^6-1=h(x)g(x)=&[1+(\beta^{14}+u_1+u_2)x+x^2]\\
		&[1+(\beta^{14}+u_1+u_2)x+(\beta^{14}+u_1+u_2)x^3+x^4] \,\, \in R_{1,3}[x;\theta].
		\end{align*}
		Since $g(x)=1+(\beta^{14}+u_1+u_2)x+(\beta^{14}+u_1+u_2)x^3+x^4$  is a palindromic polynomial with even degree then $C=( g(x))$ is a reversible DNA-code over $R_{1,3}$.
			\end{example}
			
\section{Conclusion}	

In this study, we solve the reversibility problem for DNA codes over non-chain ring $R_{k,s}$. 
%**We use skew cyclic codes that are more suitable for DNA codes because of DNA codes can be generated by ideals that is more algebraically than some previous DNA paper which use generation sets with special rules.**
We use skew cyclic codes which provide a direct solution for finding reversible DNA codes. 
This method that is based on non-commutativity property of the ring has proven to be more compact and more feasible than the commutative counterparts (e.g. \cite{agul,ise,Gse}).
%Because This method is more algebraic than  some previous DNA papers (e.g. \cite{agul,ise,Gse}) which use generation sets with special rules.
Generalization of $R_{k,s}=\mathbb{F}_{4^{2k}}[u_1,\ldots,u_{s}]/\langle u_1^2-u_1,\ldots, u_s^2-u_s\rangle$	by changing powers of $u_i$ $(i=1,2,...s)$ and obtaining optimal codes are open problems.
			There are just a few papers on k-mer based  DNA codes including this one. Investigating the reversibility problem over different rings  is an interesting problem. Further, relating these theoretical findings to real DNA data is even more interesting and challenging problem.

\end{document}